\documentclass[12pt]{amsproc}
\newtheorem{theorem}{\sc Theorem}[section]
\newtheorem{lemma}[theorem]{\sc Lemma}
\newtheorem{proposition}[theorem]{\sc Proposition}

\begin{document}
\title[Profinite groups]{On profinite groups admitting a word with only few values}
\author{Pavel Shumyatsky }
\address{ Pavel Shumyatsky: Department of Mathematics, University of Brasilia,
Brasilia-DF, 70910-900 Brazil}
\email{pavel@unb.br}
\thanks{This research was supported by the Conselho Nacional de Desenvolvimento Cient\'{\i}fico e Tecnol\'ogico (CNPq),  and Funda\c c\~ao de Apoio \`a Pesquisa do Distrito Federal (FAPDF), Brazil.}


\begin{abstract} 
A group-word $w$ is called concise if the verbal subgroup $w(G)$ is finite whenever $w$ takes only finitely many values in a group $G$. It is known that there are words that are not concise. The problem whether every word is concise in the class of profinite groups remains wide open. Moreover, there is a conjecture that every word $w$ is strongly concise in profinite groups, that is, $w(G)$ is finite whenever $G$ is a profinite group in which $w$ takes less than $2^{\aleph_0}$ values. In this paper we show that if the word $w$ takes less than $2^{\aleph_0}$ values in a profinite group $G$ then $w(w(G))$ is finite.
\end{abstract}

\maketitle

\section{Introduction} A group-word $w=w(x_1,\dots, x_t)$ in variables $x_1,\dots, x_t$ is a nontrivial element of the free group on $x_1,\dots, x_t$.
Given such a word, we think of it primarily as a function of $t$ variables defined on any given group $G$. We denote by $w(G)$ the verbal subgroup of $G$ generated by the values of $w$. 
The word $w$ is called concise in the class of groups $\mathcal C$ if the verbal subgroup $w(G)$ is finite whenever $w$ takes only finitely many values in a group $G\in\mathcal C$. In the sixties Hall raised the problem whether every word is concise in the class of all groups but in 1989 S. Ivanov \cite{Iva89} solved the problem in the negative.

The problem whether all words are concise in profinite groups remains wide open (cf. Segal \cite[p.\ 15]{Seg09} or Jaikin-Zapirain \cite{Jai08}). Throughout this paper if $G$ is a profinite group, $w(G)$ stands for the closed subgroup topologically generated by $w$-values in $G$. In recent years several new positive results with respect to this problem were obtained (see \cite{AcSh14, acc23, DeMoShEng1, DeMoShEng2, DeMoSh19,DeMoSh23Isr, FASh18, GuSh15}). 

A natural variation of the notion of conciseness for profinite groups was introduced in \cite{dks20}. The reader is referred to \cite{rz} or \cite{wi} for background information on profinite groups. Recall that the cardinality of any infinite profinite group is at least $2^{\aleph_0}$ (continuum). The word $w$ is strongly concise in  profinite groups if the verbal subgroup $w(G)$ is finite in any profinite group $G$ in which $w$ takes less than $2^{\aleph_0}$ values.  A number of new results on strong conciseness of group-words can be found in \cite{AzSh21, Det23, dks20, HPS23, KhSh23, pishu}. 

The main result of this paper is the following theorem.

\begin{theorem}\label{main} If a word $w$ takes less than $2^{\aleph_0}$ values in a profinite group $G$, then $w(w(G))$ is finite.\end{theorem}

Thus, the theorem says that the verbal subgroup of the verbal subgroup of $G$ is finite. For example, if $G$ is a profinite group such that the cardinality of the set $\{g^n;\ g\in G\}$ is less than $2^{\aleph_0}$, then $(G^n)^n$ is finite. In particular, $G^{n^2}$ is finite.

We deduce Theorem \ref{main} from a result that is interesting in its own right. Namely, it turns out that every word is strongly concise in the class of so-called $FC$-generated profinite groups (Proposition \ref{main2}). This class of groups is introduced in the next section.

\section{Preliminaries} 

If $X$ is a subset of a group $G$, we write $\langle X\rangle$ to denote the subgroup generated by $X$. For subsets $X,Y$ of a group, $[X,Y]$ denotes the subgroup generated by all commutators $[x,y]$, where $x$ and $y$ range over $X$ and $Y$ respectively. It is well-known that $[X,Y]$ is normal in $\langle X,Y\rangle$. We write $x^G$ for the conjugacy class of an element $x\in G$. If $K$ is a subgroup of $G$, the normal subgroup generated by $K$ is denoted by $K^G$.

Throughout, by a subgroup of a profinite group we mean a closed subgroup. In particular, $\langle X\rangle$ stands for the subgroup topologically generated by $X$. In some instances we also deal with abstract subgroups of profinite groups. Whenever this is the case we explicitly mention that we consider an abstract subgroup.

Given a group $G$, we write $FC(G)$ for the $FC$-centre of $G$, that is, the set of all elements $g\in G$ such that $|g^G|<\infty$. The group $G$ is an $FC$-group if $G=FC(G)$. Note that if $G$ is a profinite group, then $FC(G)$ need not be closed so in this case we treat $FC(G)$ as merely a subset of $G$ rather than a subgroup. The following properties of (abstract) $FC$-groups are well-known (see for example \cite[14.5.7, 14.5.9]{rob}). 

\begin{lemma}\label{fc} Let $G$ be an $FC$-group. Then we have
\begin{enumerate}
\item The commutator subgroup $G'$ is locally finite;
\item If $g\in G$, then $[G,g]$ is finite;
\item If $g$ is a torsion element of $G$, then $\langle g^G\rangle$ is finite.
\end{enumerate}
\end{lemma}

We say that a profinite group $G$ is $FC$-generated if $G=\langle FC(G)\rangle$. In other words, $G$ is $FC$-generated if $FC(G)$ is dense in $G$.

\begin{lemma}\label{fcgen} Let $G$ be an $FC$-generated profinite group. Then we have
\begin{enumerate}
\item If $g\in FC(G)$, then $[G,g]$ is finite;
\item If $g$ is a torsion element of $FC(G)$, then $\langle g^G\rangle$ is finite.
\end{enumerate}
\end{lemma}
\begin{proof} This is straightforward from Lemma \ref{fc} keeping in mind that $FC(G)$ is dense in $G$.
\end{proof}

The following results will be helpful. 

\begin{proposition}[\cite{dks20} Proposition 2.1]\label{klopsch}
    Let $\varphi:X\to Y$ be a continuous map between two non-empty profinite spaces that is nowhere locally constant (i.e. there is no non-empty open subset $U\subseteq X$ where $\varphi|_U$ is constant). Then $|\varphi(X)|\geq 2^{\aleph_0}$.
\end{proposition}

\begin{lemma}[\cite{dks20} Lemma 2.2]\label{fcword}
Let $G$ be a profinite group and $g\in G$ be an element whose conjugacy class $g^G$ contains less than $2^{\aleph_0}$ elements. Then $g\in FC(G)$.
\end{lemma}

\section{Main theorem}

A group-word $w$ is a commutator word if the sum of the exponents of each variable involved in $w$ is zero.

\begin{proposition}\label{main2} Every word $w$ is strongly concise in the class of $FC$-generated profinite groups.
\end{proposition}
\begin{proof} Let $G$ be an $FC$-generated profinite group and suppose that a word $w=w(x_1,\dots,x_t)$ takes less than $2^{\aleph_0}$ values in $G$. For every $w$-value $s\in G$ let $X_s$ denote the set of all tuples $$(g_1,\dots,g_t)\in G\times\dots\times G\ \ \  (t \text{ factors})$$ such that $w(g_1,\dots,g_t)=s$. By Proposition \ref{klopsch}, at least one of the sets $X_s$ contains a non-empty open subset. Hence, there are elements $a_1,\dots,a_t\in G$ and an open normal subgroup $N$ of $G$ such that for any $g_1,\dots,g_t\in N$ we have $w(a_1g_1,\dots,a_tg_t)\in X_{b_0}$, where $b_0=w(a_1,\dots,a_t)$. In other words, $$w(a_1g_1,\dots,a_tg_t)=b_0$$ for any $g_1,\dots,g_t\in N$.

Since $FC(G)$ is dense in $G$, each coset $a_iN$ contains $FC$-elements. So without loss of generality we can assume that $a_1,\dots,a_t\in FC(G)$. Furthermore, again because $FC(G)$ is dense in $G$, there are finitely many elements $b_1,\dots,b_k\in FC(G)$ such that $G=\langle b_1,\dots,b_k, N\rangle$.

We now consider the case where $w$ is a non-commutator word. So there is a variable $x_i$ such that the sum of the exponents of $x_i$ in $w$ is $r\neq 0$. Substitute the unit for all the variables except $x_i$ and an arbitrary element $g\in G$ for $x_i$. We see that $g^r$ is a $w$-value for all $g\in G$. Hence  $G$ contains less than $2^{\aleph_0}$ $r$-th powers. It follows that $G$ is torsion. 

Set $U=\langle a_1,\dots,a_t, b_0,b_1,\dots,b_k\rangle^G$. In view of Lemma \ref{fcword} $b_0\in FC(G)$. Keeping in mind that also the elements $a_1,\dots,a_t, b_1,\dots,b_k$ are contained in $FC(G)$ we deduce from Lemma \ref{fcgen} (2) that $U$ is finite. Taking into account that $G=\langle b_1,\dots,b_k, N\rangle$ and $w(a_1g_1,\dots,a_tg_t)=b_0$ for any $g_1,\dots,g_t\in N$, it follows that $G/U$ satisfies the identity $w(x_1,\dots,x_t)\equiv1$. Hence, $w(G)\leq U$. In particular, $w(G)$ is finite.

It remains to handle the case in which $w$ is a commutator word. Set $V=\prod_{1\leq i\leq t}[G,a_i]\prod_{1\leq j\leq k}[G,b_j]$. In view of Lemma \ref{fcgen} (1) the factors $[G,a_i]$ and $[G,b_j]$ are finite so $V$ is a finite normal subgroup. Let $\bar G=G/V$ and if $X\subset G$, let $\bar X$ be the image of $X$ in $\bar G$. Observe that $\bar{a}_i,\bar{b}_j\in Z(\bar{G})$. Since $w(a_1,\dots,a_t)=b_0$ and since $w$ is a commutator word, we deduce that $\bar{b}_0$ is trivial.
Taking into account that $\bar{a}_i\in Z(\bar{G})$ and $w(\bar{a}_1\bar{g}_1,\dots,\bar{a}_t\bar{g}_t)=1$ for any $\bar{g}_1,\dots,\bar{g}_t\in\bar{N}$ write $$1=w(\bar{a}_1\bar{g}_1,\dots,\bar{a}_t\bar{g}_t)=w(\bar{a}_1,\dots,\bar{a}_t)w(\bar{g}_1,\dots,\bar{g}_t)=w(\bar{g}_1,\dots,\bar{g}_t).$$ Therefore $w(\bar{N})=1$. Observe that $\bar{G}=Z(\bar{G})\bar{N}$. For arbitrary elements $\bar{h}_1,\dots,\bar{h}_t\in\bar{G}$ there are $z_i\in Z(\bar{G})$ and $\bar{g}_i\in\bar{N}$ such that $\bar{h}_i=z_i\bar{g}_i$, whence
$$w(\bar{h}_1,\dots,\bar{h}_t)=w(\bar{g}_1,\dots,\bar{g}_t)=1.$$ 

This shows that $w(\bar{G})=1$. Hence $w(G)\leq V$ is finite. The proof is complete.
\end{proof}

Theorem \ref{main} is now straightforward. Indeed, assume that the word $w$ takes less than $2^{\aleph_0}$ values in a profinite group $G$. Set $K=w(G)$. Lemma \ref{fcword} immediately shows that the $w$-values are contained $FC(G)$. Hence, $K$ is $FC$-generated. According to Proposition \ref{main2} $w(K)$ is finite. The desired result follows.

\end{document}